\documentclass[english,11pt]{article}
\usepackage[english]{babel}
\usepackage[normalem]{ulem}
\usepackage{amssymb,amsmath,amsthm,mathtools,braket,aligned-overset}
\usepackage{thmtools,thm-restate}
\usepackage{graphicx,xcolor}
\usepackage{float}
\usepackage{caption,subcaption}
\usepackage{tikz}
\usetikzlibrary{calc,decorations.pathmorphing,decorations.text,decorations.markings,matrix}

\usepackage{array,booktabs}
\usepackage{enumerate}

\usepackage[linesnumbered,ruled,vlined]{algorithm2e}

\usepackage{appendix}
\usepackage{titlesec}
\usepackage{authblk}

\usepackage[hidelinks]{hyperref}
\usepackage[capitalise,nameinlink,noabbrev]{cleveref}
\usepackage[noadjust]{cite}

\theoremstyle{plain}
\newtheorem{thm}{Thm}[section]

\newtheorem{claim}{Claim}[section]
\newtheorem{theorem}[thm]{Theorem}
\newtheorem{lemma}[thm]{Lemma}

\newtheorem{proposition}[thm]{Proposition}
\newtheorem{conjecture}[thm]{Conjecture}
\newtheorem{problem}[thm]{Problem}

\newenvironment{claimproof}{\noindent \emph{Proof of the claim.}}{\hfill$\qed$}

\newcommand{\Z}{\mbox{$\mathbb Z$}}

\allowdisplaybreaks[4]

\begin{document}

\title{Nowhere-zero $3$-flows in graphs with forbidden edge-cuts}

 \date{}
\author[ ]{Jiaao Li}
\author[ ]{Xinyuan Li}

\affil[ ]{\small School of Mathematical Sciences and LPMC, Nankai University, Tianjin 300071, China}
\affil[ ]{\small Emails: lijiaao@nankai.edu.cn; xinyuanli@mail.nankai.edu.cn; }

\maketitle

\begin{abstract}
Tutte's $3$-flow conjecture asserts that every $4$-edge-connected graph admits a nowhere-zero $3$-flow. In 2013, Lov\'{a}sz, Thomassen, Wu, and Zhang proved that every odd-$7$-edge-connected graph admits a nowhere-zero $3$-flow; consequently, the conjecture holds for $4$-edge-connected graphs with no edge-cut of size $5$. We consider the complementary situation in which $5$-edge-cuts are allowed. 
We show that, when edge-cuts of size $5$ are permitted, Tutte's $3$-flow conjecture holds for graphs with no edge-cut of any size from $6$ to $k$, where $k$ is an absolute constant. In fact, $k=40$ suffices and we prove a stronger version in which only nontrivial edge-cuts of those sizes are forbidden. A graph is called essentially $t$-edge-connected if  deleting any set of at most $t-1$ edges leaves at most one nontrivial component. Motivated by Jaeger's weak $3$-flow conjecture, we prove an analogous result for essential edge connectivity: every $4$-edge-connected, essentially $41$-edge-connected graph admits a nowhere-zero $3$-flow.

\end{abstract}

\section{Introduction}

Throughout this paper, graphs are finite and loopless, but may contain
parallel edges.  We follow \cite{BondyMurty2008} for terminology and
notation not defined here.  Let $G$ be a graph and let $D$ be an
orientation of $G$.  For a vertex $v\in V(G)$, let $E_D^+(v)$ and
$E_D^-(v)$ denote the sets of edges directed away from and toward $v$,
respectively.  For an integer $k\ge 2$, a nowhere-zero $k$-flow of $G$
is a pair $(D,f)$, where $f:E(G)\longrightarrow\{1,2,\ldots,k-1\}$
satisfies
\[
 \sum_{e\in E_D^+(v)}f(e)
 =
 \sum_{e\in E_D^-(v)}f(e)
 \qquad\text{for every }v\in V(G).
\]
An orientation $D$ of $G$ is a modulo $3$-orientation if
\[
 d_D^+(v)-d_D^-(v)\equiv 0\pmod 3
 \qquad\text{for every }v\in V(G),
\]
where $d_D^+(v)=|E_D^+(v)|$ and $d_D^-(v)=|E_D^-(v)|$. The subscripts are omitted when no ambiguity arises. 
It is well known that a graph admits a nowhere-zero $3$-flow if and
only if it admits a modulo $3$-orientation (see \cite{Tutte1966, LovaszThomassenWuZhang2013}).  We use these two formulations
interchangeably.

Integer flows were introduced by Tutte as a dual counterpart of map
coloring \cite{Tutte1954,Tutte1966}.  One of the central open problems in the area is the
following conjecture.

\begin{conjecture}[Tutte's $3$-flow conjecture]\label{CONJ1.1}
Every $4$-edge-connected graph admits a nowhere-zero $3$-flow.
\end{conjecture}

Jaeger \cite{Jaeger1979} proposed the weak $3$-flow conjecture,
asserting that there exists an integer $k$ such that every
$k$-edge-connected graph admits a nowhere-zero $3$-flow. Kochol \cite{Kochol2001} proved that \Cref{CONJ1.1} is equivalent to its restriction to $5$-edge-connected graphs.  Thomassen
\cite{Thomassen2012} proved that $k=8$ suffices in the weak conjecture.  Lov\'{a}sz,
Thomassen, Wu, and Zhang \cite{LovaszThomassenWuZhang2013} subsequently refined
the partial flow-extension method and proved the following stronger
odd-edge-connectivity result.

\begin{theorem}[Lov\'{a}sz--Thomassen--Wu--Zhang
\cite{LovaszThomassenWuZhang2013}]\label{thm:ltwz}
Every odd-$7$-edge-connected graph admits a nowhere-zero $3$-flow.
\end{theorem}

Here a graph is odd-$t$-edge-connected if every edge-cut of odd
size has at least $t$ edges. In this paper, an edge-cut is  an inclusion minimal edge set whose removal  increases the number of connected components; such a set is also called a bond in the literature. Clearly, a $4$-edge-connected graph
with no edge-cut of size $5$ is odd-$7$-edge-connected, and hence has a
nowhere-zero $3$-flow by \Cref{thm:ltwz}.  Thus the first case of
Tutte's conjecture not covered by Theorem~\ref{thm:ltwz} naturally
arises when $5$-edge-cuts are present.  This leads to the following
question: if edge-cuts of size $5$ are permitted, can the conjecture be
verified by excluding edge-cuts in a finite interval immediately above
$5$?

Our first main result provides an affirmative answer and identifies an
explicit interval of forbidden cut sizes.

\begin{theorem}\label{thm:forbidden-cuts}
Every $4$-edge-connected graph without an edge-cut of any size from
$6$ to $40$ admits a nowhere-zero $3$-flow.
\end{theorem}

The upper endpoint $40$ is an absolute constant and is not asserted to
be optimal.  Moreover, the proof yields a stronger structural
form in which only essential edge-cuts in this interval need to be
excluded, while trivial edge-cuts are allowed.  We use the following standard terminology. A component is nontrivial if it contains an edge.
An edge-cut $F$ of a connected graph $G$ is essential if
$G-F$ has at least two nontrivial components.  The graph $G$ is
essentially $t$-edge-connected if it has no essential edge-cut
with fewer than $t$ edges.  Equivalently, deleting any set of at most
$t-1$ edges leaves at most one nontrivial component.  Thus essential
edge-connectivity controls separations between edge-containing nontrivial
subgraphs while allowing small cuts that only separate trivial components. 

In analogy to Jaeger's weak $3$-flow conjecture, we seek a constant $k$ such that every  $4$-edge-connected essentially $k$-edge-connected graph admits a nowhere-zero $3$-flow. Our next theorem shows that $k=41$ suffices.

\begin{theorem}\label{thm-3-flow-thm}
    Every $4$-edge-connected essentially $41$-edge-connected graph admits a nowhere-zero $3$-flow.
\end{theorem}

We remark that the $4$-edge-connectivity assumption in
Theorems \ref{thm:forbidden-cuts} and \ref{thm-3-flow-thm} cannot be replaced by
$3$-edge-connectivity. To see this, for $t\ge 4$, let $K_{3,t}^{+}$ be obtained from the complete bipartite graph $K_{3,t}$  by adding an edge between
two vertices of degree $t$ in $K_{3,t}$.  Then $K_{3,t}^{+}$ is $3$-edge-connected and essentially $(t+1)$-edge-connected, and it has no edge-cut of any size from $4$ to $t-1$; however, it does not admit a nowhere-zero $3$-flow (see \cite{LiMaShiWangWu2022,HanLaiLi2018}).

Both Theorems \ref{thm:forbidden-cuts} and \ref{thm-3-flow-thm} are derived by means of the following technical flow extension theorem. This theorem greatly simplifies the treatment of contractible configurations in the inductive argument. A preorientation at a vertex $z_0$ is an orientation of all edges incident with $z_0$. A preorientation at a vertex is called
\emph{valid} if its imbalance is zero modulo $3$.

\begin{theorem}[Extension Theorem]\label{thm:extension}
Let $k=41$. Assume that $G$ is a $4$-edge-connected, essentially $k$-edge-connected graph, and let $z_0\in V(G)$. A preorientation at $z_0$ extends to a modulo $3$-orientation of the entire graph $G$ if the following conditions hold:
\begin{enumerate}[(i)]
        \item $d(z_{0})\leq k+1$ and the preorientation at $z_0$ is valid (i.e., $d^{+}(z_{0})-d^{-}(z_{0})\equiv 0\pmod{3}$);\label{cond1}
        \item $G-z_{0}$ is $2$-edge-connected.\label{cond2}
    \end{enumerate}
\end{theorem}

We shall prove \Cref{thm:extension} by using a minimum counterexample argument. We first establish some basic properties of this counterexample $G$. Then we try to completely lift all possible vertices of degree $4$ to show $G$ has minimum degree at least $5$. Then we use some density bound and some counting argument to  guarantee the existence of contractible configurations in $G-z_0$, proving the theorem.

The remainder of this paper is organized as follows. Section~\ref{sec:preliminary} collects the notation and preliminary results used in the proof, and then we  present the proof of Theorems \ref{thm:forbidden-cuts} and \ref{thm-3-flow-thm} with the assistance of \Cref{thm:extension}. Section \ref{sec:pf1.5} proves \cref{thm:extension} from which the main results follow. Section \ref{sec4} discusses some possible improvements and open problems. Appendix \ref{appendix-20-23} provides improved bounds through different methods.

\section{Preliminaries}
\label{sec:preliminary}

For disjoint vertex sets $X,Y\subseteq V(G)$, let $E_G(X,Y)$ denote the set of edges with one end in $X$ and the other in $Y$, and let $e_G(X,Y)=|E_G(X,Y)|$. We use $e_G(x,y)$ when $X=\{x\}$ and $Y=\{y\}$. For $X\subseteq V(G)$, let
\[
 d_G(X)=e_G(X,V(G)\setminus X).
\]
In particular, $d_G(v)=d_G(\{v\})$. We omit the subscript $G$ when the graph is clear. For a nonempty set $X\subseteq V(G)$, let $G/X$ denote the graph obtained from $G$ by identifying all vertices of $X$ as one vertex and deleting the resulting loops. All parallel edges are retained. If $H$ is a subgraph of $G$, then $G/H$ denotes $G/V(H)$. For a positive integer $t$, the notation $tK_2$ denotes the graph on two vertices with $t$ parallel edges.

A mapping $\beta:V(G)\to\Z_3$ is a \emph{$\Z_3$-boundary} if
\[
 \sum_{v\in V(G)}\beta(v)=0\quad\text{in }\Z_3.
\]
An orientation $D$ of $G$ is a \emph{$\beta$-orientation} if
\[
 d_D^+(v)-d_D^-(v)=\beta(v)\quad\text{in }\Z_3
 \qquad\text{for every }v\in V(G).
\]
The graph $G$ is \emph{$\Z_3$-connected} if it has a $\beta$-orientation for every $\Z_3$-boundary $\beta$. The concept of group connectivity was introduced by Jaeger, Linial, Payan, and Tarsi~\cite{JaegerLinialPayanTarsi1992}, serving as contractible configurations for nowhere-zero flows. In particular, every $\Z_3$-connected graph admits a modulo $3$-orientation. A crucial ingredient of group connectivity is the following proposition.

\begin{proposition}\label{prop-extension} (see \cite{Lai2000, LovaszThomassenWuZhang2013})
   Let $G$ be a graph with a  $\Z_3$-connected subgraph $H$. Let $\beta$ be a $\Z_3$-boundary of $G$. Denote $G'=G/H$ and let $w$ be the vertex of $G'$ formed by contracting $H$. Define $\beta'$ for $G'$ as $\beta'(v)=\beta(v)$ for each $v\in V(G')\setminus\{w\}$, and $\beta'(w)=\sum_{v\in V(H)} \beta(v)$.  Then any $\beta'$-orientation of $G'$ can be extended to a $\beta$-orientation of $G$. Moreover, $G$ has a modulo $3$-orientation if and only if $G/H$ does; similarly, $G$ is $\Z_3$-connected if and only if $G/H$ is.
\end{proposition}

\subsection{Some fundamental results}

Let us recall some fundamental results. 

\begin{proposition}\label{lem:basic-facts}
The following statements hold.
\begin{enumerate}[(i)]
 \item For every integer $t\geq2$, the graph $tK_2$ is $\Z_3$-connected.
 \item(see~\cite{ChenEschenLai2008,ZhangYin2013}) The graphs $K_{4,4}$ and $K_{4,4}^{-}$ are $\Z_3$-connected, where $K_{4,4}^{-}$ is obtained from $K_{4,4}$ by deleting one edge.
\end{enumerate}
\end{proposition}

The following theorem is useful when the edge-connectivity  of the graph is large.

\begin{theorem}[Lov\'{a}sz--Thomassen--Wu--Zhang~\cite{LovaszThomassenWuZhang2013}]\label{thm:ltwz-six}
Every $6$-edge-connected graph is $\Z_3$-connected.
\end{theorem}

We say that a graph is \emph{$\Z_3$-reduced} if it contains no nontrivial $\Z_3$-connected subgraph. In our proof, the following density bound, due to~\cite{LiMaShiWangWu2022},  ensures the existence of sufficiently many vertices of large degree in a minimum counterexample.

\begin{theorem}[Li--Ma--Shi--Wang--Wu~\cite{LiMaShiWangWu2022}]\label{thm:z3-density}
Let $G$ be a connected $\Z_3$-reduced graph on $n\geq4$ vertices. Then
\[
 |E(G)|\leq4n-10.
\]
Equality holds only when $G\cong K_4$. Consequently, if $G\not\cong K_4$, then $|E(G)|\leq4n-11$.
\end{theorem}

The next proposition asserts that if a graph has sufficiently high essential edge-connectivity, then the sum of the degrees of any two adjacent vertices is large. 

\begin{proposition}\label{prop:degree-sum}
Assume that $G$ is a connected, essentially $k$-edge-connected graph with $|V(G)|\geq4$, and that the underlying simple graph of $G$ is not a star. If $u$ and $v$ are adjacent vertices of $G$, then
\[
 d_G(u)+d_G(v)\geq k+2.
\]
\end{proposition}

\begin{proof}
Suppose otherwise that $u$ and $v$ are adjacent and $d_G(u)+d_G(v)\leq k+1$. Since $e_G(u,v)\geq1$,
\[
 d_G(\{u,v\})=d_G(u)+d_G(v)-2e_G(u,v)\leq k-1.
\]
If $G-\{u,v\}$ contains an edge, deleting the edges between $\{u,v\}$ and its complement leaves at least two nontrivial components. This contradicts essential $k$-edge-connectivity.

Hence $V(G)\setminus\{u,v\}$ is independent. Every edge of $G$ is incident with $u$ or $v$. Since the underlying simple graph of $G$ is not a star, it has a matching $\{uw,vz\}$ of size two. Delete every edge incident with $u$ or $v$ except one copy of $uw$ and one copy of $vz$. At most $d_G(u)+d_G(v)-2\leq k-1$ edges are deleted, and two nontrivial components remain. This is again a contradiction.
\end{proof}

For a vertex set $X\subseteq V(G)$, recall that $G/X$ is obtained from $G$ by identifying all vertices of $X$ as one vertex and deleting the resulting loops. Ordinary and essential edge-connectivity are preserved under contraction.

\begin{proposition}\label{prop:contraction-connectivity}
Let $G$ be a $t$-edge-connected, essentially $k$-edge-connected graph, and let $\varnothing\neq X\subsetneq V(G)$. Then $G/X$ is also $t$-edge-connected and essentially $k$-edge-connected. Furthermore, if $X$ is connected, then every edge-cut of $G/X$
lifts to an edge-cut of $G$ of the same size.
\end{proposition}

\begin{proof}
Let $F$ be an edge set whose deletion disconnects $G/X$. The corresponding edges also disconnect $G$, and hence $|F|\geq t$. Thus $G/X$ is $t$-edge-connected. If $(G/X)-F$ has two nontrivial components, the corresponding vertex sets in $G-F$ each contain an edge. Thus $G-F$ has at least two nontrivial components, and hence $|F|\geq k$. Therefore, $G/X$ is essentially $k$-edge-connected.

The proof of the second part is simple and we omit it.
\end{proof}

The following lemma presents a method for finding a $\Z_3$-connected subgraph $K_{4,4}$ or $K_{4,4}^{-}$ in a dense simple bipartite graph.

\begin{lemma}\label{lem:common-neighbors}
Let $H$ be a simple bipartite graph with parts $A$ and $B$. Let $x,y,z$ be positive integers. Suppose that every vertex of $A$ has at least $x$ neighbors in $B$, where $1\leq y\leq x\leq|B|$. If
\[
 |A|>\frac{(z-1)\binom{|B|}{y}}{\binom{x}{y}},
\]
then some $z$ vertices of $A$ have at least $y$ common neighbors in $B$.
\end{lemma}

\begin{proof}
We use a double counting argument. Suppose to the contrary that no $z$ vertices of $A$ have $y$ common neighbors in $B$. Count the pairs $(w,C)$ such that $w\in A$ and $C$ is a $y$-subset of $N_H(w)$. Each $y$-subset of $B$ is contained in the neighborhoods of at most $z-1$ vertices of $A$. Therefore,
\[
 \sum_{w\in A}\binom{d_H(w)}{y}\leq(z-1)\binom{|B|}{y}.
\]
On the other hand, given a vertex $w\in A$, there are at least $\binom{x}{y}$ number of $y$-subsets in $N_H(w)$. Thus,
\[
 \sum_{w\in A}\binom{d_H(w)}{y}\geq|A|\binom{x}{y}>(z-1)\binom{|B|}{y},
\]
a contradiction.
\end{proof}

\subsection{Proofs of Theorems \ref{thm:forbidden-cuts} and \ref{thm-3-flow-thm} from the Extension Theorem}
We first derive Theorem \ref{thm-3-flow-thm} from Theorem \ref{thm:extension}, and then prove Theorem  \ref{thm:forbidden-cuts} via Theorems \ref{thm-3-flow-thm} and \ref{thm:extension}.

\begin{proof}[Proof of Theorem \ref{thm-3-flow-thm} under the assumption of Theorem \ref{thm:extension}]
Let $k=41$, and let $G$ be a $4$-edge-connected essentially $k$-edge-connected graph.
If $|V(G)|=2$, then the graph $G$ is isomorphic to $tK_{2}$ for $t\geq4$, and so $G$ is $\Z_{3}$-connected by Proposition \ref{lem:basic-facts}, and hence admits a nowhere-zero $3$-flow.

Suppose that $|V(G)|=3$, say $V(G)=\{u,v,w\}$. Since $G$ is $4$-edge-connected, we can assume $e(u,v)\geq2$ after possible relabeling.
The subgraph formed by $G[\{u,v\}]$ is $\Z_3$-connected by Proposition \ref{lem:basic-facts}.
Then we also have $G/\{u,v\}$ is isomorphic to $tK_{2}$ for $t\geq2$ and thus $\Z_{3}$-connected by Proposition \ref{lem:basic-facts}. It follows from Proposition \ref{prop-extension} that $G$ is $\Z_3$-connected and admits a nowhere-zero $3$-flow.

     Assume $|V(G)|\geq4$. If $\delta(G)\geq 6$, then $G$ is $6$-edge-connected and admits a nowhere-zero $3$-flow by Theorem \ref{thm:ltwz-six}. Hence $\delta(G)\leq5$. Take $z_{0}\in V(G)$ with $d(z_{0})=\delta(G)$. Fix a valid preorientation at $z_0$ such that $d^{+}(z_{0})-d^{-}(z_{0})\equiv 0\pmod{3}$. Now we only need to verify (\ref{cond2}) to apply Theorem \ref{thm:extension}. If $G-z_{0}$ is disconnected, then one of its components would be incident with at most $\lfloor\frac{d(z_{0})}{2}\rfloor\leq2$ at $z_0$. These edges would form an edge-cut of size at most $2$, contrary to the $4$-edge-connectivity of $G$. If $G-z_{0}$ is connected but has a bridge $e$, then we write $A,B$ for the two connected components of $G-z_{0}-e$. Without loss of generality we assume $|A|\geq2$ since $|V(G)|\geq4$. The subgraph $G[B\cup \{z_0\}]$ also contains an edge. Now $|B\cup \{z_{0}\}|\geq2$ and we have $e(A,B\cup \{z_{0}\})\geq k$. Moreover, we have $e(B,B^{c})\geq 4$, and
     \[
     d(z_{0})=d_A(z_{0})+d_B(z_{0})=e(A,B\cup \{z_{0}\})+e(B,B^{c})-2\geq k+2>5\geq d(z_0),
     \]
     which is a contradiction. Hence $G-z_{0}$ is $2$-edge-connected and we can apply Theorem \ref{thm:extension} to extend the preorientation of $E(z_{0})$ to a modulo $3$-orientation of $G$.
\end{proof}
Now we prove Theorem  \ref{thm:forbidden-cuts} via Theorems \ref{thm-3-flow-thm} and \ref{thm:extension}. In fact, we prove the following stronger form, which implies Theorem \ref{thm:forbidden-cuts}.
\begin{theorem}
    Every $4$-edge-connected graph without an essential edge-cut of any size from
$6$ to $40$ admits a nowhere-zero $3$-flow.
\end{theorem}
\begin{proof}
    Suppose to the contrary that $G$ is a counterexample with $|V(G)|$ minimum.

    If $G$ is $4$-edge-connected essentially $41$-edge-connected, then by Theorem \ref{thm-3-flow-thm}, $G$ admits a nowhere-zero $3$-flow. Since $G$ contains no essential edge-cut of size from $6$ to $40$, we can assume $G$ contains an essential edge-cut $E(A,B)$ of size $4$ or $5$, where $A$ and $B$ are connected and $|A|,|B|\geq2$. We take such an essential edge-cut with $|B|$ minimum.
    
    We claim that $B$ is $2$-edge-connected. For otherwise, $B$ contains an edge $e$ such that $G[B]-e$ contains two connected components. 
    One of them is incident with $A$ by at most $\lfloor\frac{1}{2}|E(A,B)|\rfloor\leq 2$ edges.
    This implies $G$ can be made disconnected by deleting at most $3$ edges, contrary to that $G$ is $4$-edge-connected.

    Then we consider the graph $G/B$ obtained from $G$ by contracting $B$. By Proposition \ref{prop:contraction-connectivity}, $G/B$ is also $4$-edge-connected and contains no essential edge-cut of size from $6$ to $40$. Since $|V(G/B)|<|V(G)|$ and $G$ is a minimum counterexample, $G/B$ admits a modulo $3$-orientation $D'$. We regard $D'$ as a preorientation of $G$, and the orientation of $E(A,B)$ as a valid preorientation at $\{A\}$ in $G/A$. 

    Finally we consider the graph $G/A$ obtained from $G$ by contracting $A$. We have $d_{G/A}(\{A\})\leq 5\leq 42$, $B=G/A -\{A\}$ is $2$-edge-connected and the preorientation at $\{A\}$ is valid. To apply Theorem \ref{thm:extension}, it suffices to verify $G/A$ is essentially $41$-edge-connected. Note that $G/A$ contains no essential edge-cut of size from $6$ to $40$, we only need to show $G/A$ contains no essential edge-cut of size from $4$ to $5$. In fact, if $G/A$ contains an essential edge-cut $E(X,Y)$ of size $4$ or $5$ with $\{A\}\in X$, then in $G$ the edge-cut $E((X\setminus \{A\})\cup A,Y)$ is also an essential edge-cut of size $4$ or $5$ with $|Y|\leq |B|-1$ since $|X|\geq2$, contrary to the minimal choice of $|B|$. By Theorem \ref{thm:extension}, the preorientation of $E(\{A\},G/A-\{A\})$ can be extended to a modulo $3$-orientation $D''$ of $G/A$. By combining the partial orientation $D'$ and $D''$, we obtain a modulo $3$-orientation of $G$ and complete the proof.
\end{proof}




\section{Proof of the Extension Theorem}
\label{sec:pf1.5}
\crefname{claim}{Claim}{Claims}
\Crefname{claim}{Claim}{Claims}

For convenience, we restate the Extension Theorem.

\begin{theorem}[Extension Theorem]\label{thm:extension-section3}
Let $k=41$. Assume that $G$ is a $4$-edge-connected, essentially
$k$-edge-connected graph, and let $z_0\in V(G)$. A preorientation at
$z_0$ extends to a modulo $3$-orientation of the entire graph $G$ if the following
conditions hold:
\begin{enumerate}[(i)]
 \item $d(z_{0})\leq k+1$ and the preorientation at $z_0$ is valid, i.e. $d^{+}(z_{0})-d^{-}(z_{0})\equiv 0\pmod{3}$;
 \item $G-z_0$ is $2$-edge-connected.
\end{enumerate}
\end{theorem}

We prove the theorem by considering a minimum counterexample. The first
few claims establish its basic structure. We then completely lift every vertex
of degree $4$. Finally, the density bound in \cref{thm:z3-density} and
the common-neighbor argument in \cref{lem:common-neighbors} force a copy
of $K_{4,4}$ or $K_{4,4}^{-}$ in $G-z_0$.

\begin{proof}[Proof of \Cref{thm:extension-section3}]
Suppose that the theorem is false. Among all counterexamples, let
$(G,z_0)$ satisfy the following conditions:
\begin{enumerate}[(1)]
 \item $|E(G-z_0)|$ is minimum;
 \item subject to (1), $d(z_0)$ is maximum.
\end{enumerate}
We establish a sequence of claims on the properties of counterexample $G$.

\begin{claim}\label{cl:sec3-z3-reduced}
The graph $G-z_0$ is $\Z_3$-reduced. In particular, it contains no
$2K_2$ and hence is simple.
\end{claim}

\begin{claimproof}
Suppose that $H$ is a nontrivial $\Z_3$-connected subgraph of
$G-z_0$, and let $G'=G/H$. By
\cref{prop:contraction-connectivity}, the graph $G'$ remains
$4$-edge-connected and essentially $k$-edge-connected. Moreover,
contraction inside $G-z_0$ preserves the $2$-edge-connectivity of
$G-z_0$, and the preorientation at $z_0$ remains unchanged. Since
$|E(G'-z_0)|<|E(G-z_0)|$, condition (1) implies that the
preorientation extends to $G'$. By \cref{prop-extension}, this
orientation extends through $H$ to a modulo $3$-orientation of $G$, a
contradiction. The last assertion follows from part (i) of
\cref{lem:basic-facts}.
\end{claimproof}

\begin{claim}\label{cl:sec3-root-degree}
We have $d(z_0)=k+1$.
\end{claim}

\begin{claimproof}
Suppose that $d(z_0)\leq k$, and let $e=uz_0$ be a preoriented edge.
Delete $e$ and replace it with two parallel edges, both directed
opposite to $e$. At $z_0$, the total contribution of the two new edges
is congruent modulo $3$ to the contribution of $e$. The resulting
graph $G'$ satisfies the two hypotheses of the theorem. It also
satisfies
\[
 |E(G'-z_0)|=|E(G-z_0)|
 \quad\text{and}\quad
 d_{G'}(z_0)=d_G(z_0)+1\leq k+1.
\]
By condition (2) on minimal counterexample $G$, the preorientation extends to $G'$. Delete one of
the two new edges and reverse the other. The imbalance at each end
changes by a multiple of $3$, and the original edge $e$ is recovered.
Thus the preorientation extends to $G$, a contradiction.
\end{claimproof}

\begin{claim}\label{cl:sec3-order}
We have $|V(G)|\geq 6$, or equivalently,
$|V(G-z_0)|\geq 5$.
\end{claim}

\begin{claimproof}
If $|V(G)|=2$, the preorientation already orients every edge of $G$
and hence is an extension, contrary to the choice of $G$. If
$|V(G)|=3$, then by $G-z_0$ is $2$-edge-connected, it consists of two vertices joined by at least
two parallel edges. This contradicts \cref{cl:sec3-z3-reduced}.

Suppose that $|V(G)|=4$. Since $G-z_0$ is simple and
$2$-edge-connected, by \cref{{cl:sec3-z3-reduced}} it is a triangle with vertices $v_1,v_2,v_3$.
By \cref{prop:degree-sum},
$d(v_i)+d(v_{i+1})\geq k+2$, where the subscripts are taken modulo $3$.
Consequently,
\[
 d(z_0)=\sum_{i=1}^3d(v_i)-6
 \geq \frac{3}{2}(k+2)-6\geq k+2,
\]
contrary to $d(z_0)\leq k+1$.

Finally, suppose that $|V(G)|=5$. The simple
$2$-edge-connected graph $G-z_0$ contains a cycle of length $3$ or
$4$. We apply \cref{prop:degree-sum} again. A triangle gives
\[
 d(z_0)\geq\sum_{i=1}^3d(v_i)-9
 \geq \frac{3}{2}(k+2)-9\geq k+2,
\]
or a $4$-cycle gives
\[
 d(z_0)\geq\sum_{i=1}^4d(v_i)-12
 \geq 2(k+2)-12\geq k+2.
\]
Both inequalities contradict $d(z_0)\leq k+1$ when $k=41$.
\end{claimproof}

\begin{claim}\label{cl:sec3-large-cut}
If $E(X,X^c)$ is an essential edge-cut with
$\min\{|X|,|X^c|\}\geq 3$, then
\[
 e(X,X^c)\geq k+2.
\]
\end{claim}

\begin{claimproof}
Suppose instead that $e(X,X^c)\leq k+1$, and assume that
$z_0\in X^c$. By the definition of an edge-cut, both $G[X]$ and
$G[X^c]$ are connected. Each of the following contractions strictly
decreases $|E(G-z_0)|$.

First, consider $G/X$. By \cref{prop:contraction-connectivity}, this
graph remains $4$-edge-connected and essentially
$k$-edge-connected. The graph $(G/X)-z_0$ is obtained from the
$2$-edge-connected graph $G-z_0$ by a contraction and is therefore
$2$-edge-connected. Hence minimality implies that the prescribed
orientation at $z_0$ extends to $G/X$. In particular, it induces an
orientation of $E(X,X^c)$ with zero total imbalance on the $X$-side.

Next, consider $G/X^c$, with the contracted vertex as the new root
and with the induced preorientation. By
\cref{prop:contraction-connectivity}, this graph remains
$4$-edge-connected and essentially $k$-edge-connected. Its root
degree is at most $k+1$. It remains to show that deleting the new root
leaves the graph $G[X]$, which is $2$-edge-connected. Suppose that
$G[X]$ has a bridge
with shores $A$ and $B$, and let
\[
 a=e(A,X^c)
 \quad\text{and}\quad
 b=e(B,X^c).
\]
If both $A$ and $B$ contain at least two vertices, the cuts around
$A$ and $B$ are essential. Hence
$a+1\geq k$ and $b+1\geq k$, contradicting
$a+b=e(X,X^c)\leq k+1$. If, say, $A$ is a singleton, then
$a+1\geq 4$. The cut around the connected set $B$ is essential, so
$b+1\geq k$. Again $a+b\geq k+2$, a contradiction. Therefore,
$G[X]$ is $2$-edge-connected.

Minimality now gives an extension on $G/X^c$. The two orientations
have the same directions on $E(X,X^c)$, and their union is an
extension on $G$, a contradiction.
\end{claimproof}

\begin{claim}\label{cl:sec3-no-four}
No vertex of $G$ has degree $4$.
\end{claim}

\begin{claimproof}
Suppose that $w\in V(G)\setminus\{z_0\}$ has degree $4$, and let
$H=G-z_0$. Since $H$ is $2$-edge-connected,
$2\leq d_H(w)\leq 4$. By \cref{cl:sec3-z3-reduced}, the neighbors of
$w$ in $H$ are distinct. Pair the four edges incident with $w$ and
lift the two pairs. Choose the pairs so that two $wz_0$-edges
are never paired together. If $H-w$ is
disconnected, then $d_H(w)=4$ and $H-w$ has exactly two components;
in this case, choose both pairs across the two components. Let $G'$
be the resulting graph. Then $H'=G'-z_0$ is connected. If an edge
incident with $z_0$ is paired with $wu$, orient the new $z_0u$-edge
at $z_0$ in the same direction as the original $z_0w$-edge.

We first verify the connectivity properties of $G'$. Let
$F=E_{G'}(X,X^c)$ be an edge-cut of $G'$. By restoring $w$ on one
shore, we obtain an edge-cut $F^*$ of $G$ with $|F^*|=|F|$. The only
exception occurs when the two lifted edges lie internally on opposite
shores. In that case, $|F^*|=|F|+2$. In the first case, the
$4$-edge-connectivity of $G$ gives $|F|\geq 4$. Moreover, if $F$ is
essential, the shore containing $w$ can be selected so that $F^*$ is
essential. Hence $|F|\geq k$.

Consider the exceptional case, now both shores contain an edge and thus the edge-cut $F$ is essential. Suppose that $|F|<k$. Each shore
contains both ends of one lifted edge. Fix one shore $X$, delete its
lifted edge, and restore $w$ on the other shore. If the resulting
subgraph on $X$ is connected, then its boundary is an essential
edge-cut of $G$ of size $|F|+2\leq k+1$. By
\cref{cl:sec3-large-cut}, we obtain $|X|=2$. Otherwise, the lifted
edge joins exactly two components of the resulting subgraph on $X$.
The orders of the edge-cuts around these two components have sum
$|F|+2\leq k+1$. If either component contains an edge, essential
$k$-edge-connectivity and $4$-edge-connectivity imply that this sum
is at least $k+4$, a contradiction. Hence both components are
singletons, and again $|X|=2$. The same argument applied to $X^c$
gives $|X^c|=2$. Thus $|V(G')|=4$ and $|V(G)|=5$, contrary to
\cref{cl:sec3-order}. Therefore, $|F|\geq k$ in the exceptional
case as well. It follows that $G'$ is
$4$-edge-connected and essentially $k$-edge-connected.

It remains to show that $H'$ has no bridge. Suppose that a bridge of
$H'$ has shores $A$ and $B$. If $|A|,|B|\geq 2$, then
\[
 1+e(A,z_0)\geq k
 \quad\text{and}\quad
 1+e(B,z_0)\geq k,
\]
contrary to $d_{G'}(z_0)=d_G(z_0)=k+1$. If, say, $A$ is a singleton,
the $4$-edge-connectivity of $G'$ gives $e(A,z_0)\geq 3$. Unless $B$
is also a singleton, the cut around $B$ is essential, so
$e(B,z_0)\geq k-1$. Again, $d_{G'}(z_0)\geq k+2$. If both $A$ and
$B$ are singletons, then $|V(G)|=4$, contrary to
\cref{cl:sec3-order}. Hence $H'$ is
$2$-edge-connected.

The complete lifting operation decreases $|E(G-z_0)|$ by $2$ and preserves
the root preorientation. Minimality gives an extension on $G'$. When
a lifting is reversed, replace each oriented lifted edge with the
corresponding directed two-edge path through $w$. This operation
preserves the imbalance at every vertex and gives an extension on
$G$, a contradiction.
\end{claimproof}

By \cref{cl:sec3-no-four} and essential $k$-edge-connectivity, the
graph $G$ is $5$-edge-connected. Indeed, an edge-cut of order $4$ would
be nonessential and hence would isolate a vertex of degree $4$.

\begin{claim}\label{cl:sec3-three-edge}
For every $w\in V(G-z_0)$, we have
$d_{G-z_0}(w)\geq 3$. Consequently, $G-z_0$ is
$3$-edge-connected.
\end{claim}

\begin{claimproof}
Suppose first that $d_{G-z_0}(w)\leq 2$. Since $G-z_0$ is
$2$-edge-connected, equality holds. Since $G$ is
$5$-edge-connected, $e(w,z_0)\geq 3$. Contract $\{w,z_0\}$ to a new
root $w_0$. By \cref{prop:contraction-connectivity}, the contracted
graph remains $5$-edge-connected and essentially
$k$-edge-connected. Orient the two edges incident with $w_0$ that
arise from edges of $G-z_0$ incident with $w$ so that the imbalance
at $w_0$ is $0$ modulo $3$. This is possible because the possible
sums are $-2$, $0$, and $2$, which represent all residues modulo $3$.
The degree of the new root
is at most
\[
 d(z_0)-e(w,z_0)+d_{G-z_0}(w)\leq k.
\]
We show that deleting $w_0$ leaves a $2$-edge-connected graph.

The graph $G-z_0-w$ is connected. Otherwise, one of the two edges of
$G-z_0$ incident with $w$ would be a bridge. Suppose that
$G-z_0-w$ has a bridge with shores $A$ and $B$. The two edges from
$w$ in $G-z_0$ must meet different shores. If $|A|\geq 3$, then
\[
 e(A,A^c)=e(A,z_0)+2\geq k+2
\]
by \cref{cl:sec3-large-cut}. Thus $e(A,z_0)\geq k$, contrary to
$d(z_0)-e(w,z_0)\leq k-2$. The same argument applies to $B$.
Consequently, $|A|,|B|\leq 2$. By \cref{cl:sec3-order},
$|A|=|B|=2$. Essential $k$-edge-connectivity now gives
$e(A,z_0),e(B,z_0)\geq k-2$, and hence
\[
 d(z_0)\geq 2k-1>k+1,
\]
a contradiction. Therefore, $G-z_0-w$ is $2$-edge-connected.

The contraction removes the two edges of $G-z_0$ incident with $w$.
Minimality therefore gives an extension on the contracted graph. When
the contraction is reversed, the prescribed
edges at $z_0$ retain their directions. The prescribed imbalance at
$z_0$ is $0$, and the imbalance at the contracted root is also $0$.
It follows that the imbalance at $w$ is $0$. Thus the extension lifts
to $G$, a contradiction. Therefore, $d_{G-z_0}(w)\geq 3$.

Suppose next that $G-z_0$ has a $2$-edge-cut with shores $A$ and
$B$. The minimum-degree conclusion just proved gives
$|A|,|B|\geq 2$. By \cref{cl:sec3-order}, we may assume that
$|B|\geq 3$. Since $G$ is $5$-edge-connected,
$e(A,z_0)+2\geq 5$, while \cref{cl:sec3-large-cut} gives
$e(B,z_0)+2\geq k+2$. Therefore,
\[
 d(z_0)=e(A,z_0)+e(B,z_0)\geq k+3,
\]
contrary to \cref{cl:sec3-root-degree}. This proves \cref{cl:sec3-three-edge}.
\end{claimproof}

It remains to find a nontrivial $\Z_3$-connected subgraph of $G-z_0$.

\begin{claim}\label{cl:sec3-bipartite-copy}
The graph $G-z_0$ contains $K_{4,4}$ or $K_{4,4}^{-}$ as a subgraph.
\end{claim}

\begin{claimproof}
Let $n=|V(G)|$. By \cref{cl:sec3-order},
\cref{cl:sec3-z3-reduced}, and \cref{thm:z3-density},
\[
 e(G-z_0)\leq 4(n-1)-11.
\]
Together with \cref{cl:sec3-root-degree}, this gives
\begin{equation}\label{eq:sec3-degree-upper}
 \sum_{v\in V(G)}d(v)=2e(G)\leq 8n+2(k-14).
\end{equation}

Let $V_i$ be the set of vertices of degree $i$ in $G$. Define
$V_{\geq i}=\cup_{t\geq i}V_{t}$ and $V_{i\sim j}=\cup_{i\leq t\leq j}V_t$. By
\cref{prop:degree-sum}, the set $V_{\leq (k+1)/2}$ is independent.
Moreover, every neighbor of a vertex in $V_5$, $V_6$, or $V_7$ has
degree at least $k-3$, $k-4$, or $k-5$, respectively. Hence
\begin{equation}\label{eq:sec3-high-degree-sum}
 \sum_{v\in V_{\geq k-5}}d(v)
 \geq 5|V_5|+6|V_6|+7|V_7|.
\end{equation}
Combining \eqref{eq:sec3-degree-upper} and
\eqref{eq:sec3-high-degree-sum} yields
\begin{equation}\label{eq:sec3-high-count}
 \begin{split}
 8|V_{\geq k-5}|\geq{}&
 2|V_5|+4|V_6|+6|V_7|\\
 &+\sum_{v\in V_{8\sim k-6}}(d(v)-8)-2(k-14).
 \end{split}
\end{equation}
Using \eqref{eq:sec3-high-count} once more in
\eqref{eq:sec3-degree-upper}, we obtain
\begin{align}
 \frac{(k-5)(k-14)}{4}
 \geq{}&
 \frac{k-25}{4}|V_5|
 +\frac{k-17}{2}|V_6|
 +\frac{3k-43}{4}|V_7| \notag\\
 &+\frac{k-5}{8}
 \sum_{v\in V_{8\sim k-6}}(d(v)-8).
 \label{eq:sec3-weighted}
\end{align}

For $k=41$, \eqref{eq:sec3-weighted} becomes
\begin{equation}\label{eq:sec3-k41-weighted}
 \begin{split}
 243\geq{}&
 4|V_5|+12|V_6|+20|V_7|\\
 &+\frac{9}{2}\sum_{v\in V_{8\sim 35}}(d(v)-8),
 \end{split}
\end{equation}
while \eqref{eq:sec3-degree-upper} gives
\begin{align}
 3|V_5|+2|V_6|+|V_7|+54
 \geq{}&
 \sum_{v\in V_{8\sim 35}}(d(v)-8)
 +28|V_{36}|+29|V_{37}| \notag\\
 &+\sum_{v\in V_{\geq 38}}(d(v)-8).
 \label{eq:sec3-k41-degree}
\end{align}

For $i\in\{5,6,7\}$, let $V_{i,j}$ consist of the vertices
$v\in V_i$ such that $e(v,z_0)=j$, and let
$V'_i=V_i\cap V(G-z_0)$. Define
\[
 \begin{split}
 L_5&=V_{5,0}\cup V_{6,\leq 1}\cup V_{7,\leq 2},\\
 L_4&=V_{5,\leq 1}\cup V_{6,\leq 2}\cup V_{7,\leq 3},\\
 L_3&=V_{5,\leq 2}\cup V_{6,\leq 3}\cup V_{7,\leq 4}.
 \end{split}
\]
Every vertex of $L_i$ has at least $i$ distinct neighbors in
$V'_{\geq 36}$. Let
\[
 T=3|V_5|+2|V_6|+|V_7|.
\]
Since $d(z_0)=42$, the contribution of $z_0$ to the last sum in
\eqref{eq:sec3-k41-degree} is $d(z_0)-8=34$. Every vertex of
$V'_{\geq 36}$ contributes at least $28$. Therefore,
\begin{equation}\label{eq:sec3-high-set-bound}
 28|V'_{\geq 36}|\leq T+20.
\end{equation}
We consider the possible values of $T$.

If $T\geq 181$, then $|V_5|+|V_6|+|V_7|\geq 61$. The first three
terms on the right-hand side of \eqref{eq:sec3-k41-weighted} have
sum at least $244$, a contradiction.

Suppose that $170\leq T\leq 180$. Then
$|V_5|+|V_6|+|V_7|\geq 57$ and
$|L_5|\geq 57-d(z_0)=15$. By
\eqref{eq:sec3-high-set-bound}, $|V'_{\geq 36}|\leq 7$. Apply
\cref{lem:common-neighbors} first with
$(x,y,z)=(5,3,5)$ and then with $(x,y,z)=(2,1,3)$. We obtain five
vertices with three common high-degree neighbors and, among these
five vertices, three with one further common neighbor. These vertices
contain a copy of $K_{4,4}^{-}$.

Suppose that $142\leq T\leq 169$. Then there are at least $48$
vertices in $V_{5\sim 7}$ and, by
\eqref{eq:sec3-high-set-bound}, $|V'_{\geq 36}|\leq 6$. If
$|L_5|\geq 7$, two applications of \cref{lem:common-neighbors}, with
$(x,y,z)=(5,3,4)$ and $(x,y,z)=(2,1,3)$, give a copy of
$K_{4,4}^{-}$. If $|L_5|\leq 6$, at least $42$ low-degree vertices
are incident with a root edge. Since $d(z_0)=42$, every low-degree
vertex belongs to $L_4$. Thus $|L_4|\geq 48$. Applying
\cref{lem:common-neighbors} with $(x,y,z)=(4,4,4)$ gives a copy of
$K_{4,4}$.

Suppose that $114\leq T\leq 141$. There are at least $38$ vertices
in $V_{5\sim 7}$ and, by \eqref{eq:sec3-high-set-bound},
$|V'_{\geq 36}|\leq 5$. If
$|L_5|\geq 4$, these four vertices and their five high-degree
neighbors contain a copy of $K_{4,4}$. Otherwise, at least $35$
low-degree vertices are incident with a root edge. Every vertex
outside $L_4$ is incident with at least one additional root edge.
It follows that $|L_4|\geq 31$. Applying
\cref{lem:common-neighbors} with $(x,y,z)=(4,4,4)$ again gives a copy
of $K_{4,4}$.

Suppose that $86\leq T\leq 113$. There are at least $29$
low-degree vertices and, by \eqref{eq:sec3-high-set-bound},
$|V'_{\geq 36}|\leq 4$. Thus
$L_5=\varnothing$. Counting the $42$ root edges shows that at most
$42-29=13$ low-degree vertices lie outside $L_4$. Hence
$|L_4|\geq 16$. Every vertex of $L_4$ is adjacent to all four
vertices of $V'_{\geq 36}$, which gives a copy of $K_{4,4}$.

Suppose that $64\leq T\leq 85$. There are at least $22$
low-degree vertices. At most $21$ of them can lie outside $L_4$,
since each such vertex is incident with at least two root edges.
Hence $L_4\neq\varnothing$, so $|V'_{\geq 36}|\geq 4$. However,
\eqref{eq:sec3-high-set-bound} gives
$|V'_{\geq 36}|\leq 3$, a contradiction.

If $1\leq T\leq 63$, then $L_3\neq\varnothing$ by
\cref{cl:sec3-three-edge}. Thus $|V'_{\geq 36}|\geq 3$, whereas
\eqref{eq:sec3-high-set-bound} gives
$|V'_{\geq 36}|\leq 2$, a contradiction.

It remains to consider $T=0$. Then every vertex has degree at least
$8$. Let $s$ be the minimum degree in $G$ among the vertices of
$G-z_0$. Since
\[
 \sum_{v\in V(G-z_0)}d_G(v)
 =2e(G-z_0)+d(z_0)\leq 8n+12,
\]
\cref{cl:sec3-order} gives $8\leq s\leq 12$. Let $v$ be a vertex
of degree $s$. By \cref{cl:sec3-three-edge}, the vertex $v$ has
three distinct neighbors in $G-z_0$. By \cref{prop:degree-sum},
each of these neighbors has degree at least $43-s$. Therefore,
\[
 \sum_{x\in V(G)}d(x)
 \geq 42+s+3(43-s)+(n-5)s
 =171+(n-7)s>8n+54,
\]
contrary to \eqref{eq:sec3-degree-upper}. This proves the claim.
\end{claimproof}

By part (ii) of \cref{lem:basic-facts}, each subgraph in
\cref{cl:sec3-bipartite-copy} is $\Z_3$-connected. This contradicts
\cref{cl:sec3-z3-reduced} and completes the proof.
\end{proof}


\section{Remarks}
\label{sec4}

In this paper, we prove that every $4$-edge-connected essentially $41$-edge-connected graph admits a nowhere-zero $3$-flow. Here the constant 41 is not optimal and we do not attempt to optimize it, since the current proof becomes much more complicated when $k$ is slightly smaller. Moreover, there is a limitation of our method caused by the density of $\Z_3$-reduced graphs $G$.  Note that the density bound $|E(G)|\leq 4|V(G)|-10$ in \cref{thm:z3-density} is not sufficient to  prove a very small bound on essential edge connectivity. Another interesting problem is to seek a $\Z_3$-connectivity analogue of \cref{thm-3-flow-thm}.

\begin{problem}
    Is it true that there exists  a constant $k$ such that every $4$-edge-connected essentially $k$-edge-connected graph is $\Z_3$-connected ?
\end{problem}
Note that there exist $4$-edge-connected essentially $6$-edge-connected graphs that are not $\Z_3$-connected as shown in \cite{Hasanvand2023}. The problem would have an affirmative answer if one could prove a stronger conjecture that every graph with three edge-disjoint spanning trees is $\Z_3$-connected.

After completing the proofs of our main result, we used ChatGPT to explore possible improvement to the constants. This led us to distinct methods yielding  much better constants for Theorems \ref{thm:forbidden-cuts} and \ref{thm-3-flow-thm}, albeit without a uniform framework for the proofs. We include the resulting arguments in Appendix \ref{appendix-20-23}.

\section*{Declaration of AI usage}
The main body of this paper, excluding the Appendix, is entirely the authors' own work--all mathematical ideas, problem formulations, proof strategies, and arguments were conceived and developed solely by  the authors without AI-assistance. For the main text, ChatGPT was used exclusively as a proofreading aid to identify and correct typographical and linguistic errors and to improve the language of the manuscript. The Appendix, in contrast, involved AI‑assisted arguments; nevertheless, all AI‑generated content in the Appendix was carefully checked, revised, and independently verified by the authors.

\section*{Acknowledgement}
Jiaao Li and Xinyuan Li are partially supported by National Key Research and Development Program of China (No. 2022YFA1006400), National Natural Science Foundation of China (No. 12571371), Natural Science Foundation of Tianjin (No. 24JCJQJC00130), and the Fundamental Research Funds for the Central Universities, Nankai University.


\renewcommand{\appendixname}{Appendix}
\renewcommand{\thesection}{\Alph{section}} 
\appendix
\titleformat{\section}[block]{\normalfont\Large\bfseries}{\appendixname~\thesection}{1em}{}

\section[ImprovedBound]{Improved bounds on constants}
\label{appendix-20-23}

In this appendix, we improve both bounds stated in the introduction through different methods. The
first result lowers the essential edge-connectivity from $41$ to $23$.
The second result shows that it is enough to forbid edge-cuts of sizes
from $6$ to $20$.

\begin{theorem}\label{thm:essential-23}
Every $4$-edge-connected essentially $23$-edge-connected graph admits a nowhere-zero $3$-flow.
\end{theorem}

\begin{theorem}\label{thm:forbidden-20}
Every $4$-edge-connected graph without an edge-cut of any
size from $6$ to $20$ admits a nowhere-zero $3$-flow.
\end{theorem}

We first recall three tools used in the proofs. For distinct vertices
$x$ and $y$, let $\lambda_G(x,y)$ denote the maximum number of
pairwise edge-disjoint $x$--$y$ paths in $G$. Lifting two edges $sx$
and $sy$ at $s$ means deleting them and adding a new edge $xy$. A
complete lifting at a vertex of degree $4$ consists of two such
liftings followed by deletion of the isolated vertex. A vertex $s$ is
\emph{nonseparating} if $G-s$ is connected.

\begin{theorem}[Mader~\cite{Mader1978}]\label{thm:sec4-mader}
Let $s$ be a nonseparating vertex of degree $4$ in a graph $G$, and assume that $s$ has at least two distinct
neighbors. There is a complete lifting at $s$ that preserves
 $\lambda_G(x,y)$ for every two distinct vertices $x,y\neq s$.
\end{theorem}

Indeed, Mader's theorem gives the first admissible lifting. Suppressing
the remaining vertex of degree $2$ preserves the same local
edge-connectivities and gives the stated complete lifting.

\begin{theorem}[Nash-Williams--Tutte
\cite{NashWilliams1961,Tutte1961}]\label{thm:sec4-nwt}
A graph $G$ contains $t$ pairwise edge-disjoint spanning
trees if and only if
\[
 e_G(\mathcal P)\geq t(|\mathcal P|-1)
\]
for every partition $\mathcal P$ of $V(G)$, where
$e_G(\mathcal P)$ denotes the number of edges joining distinct parts.
\end{theorem}

In terms of spanning trees, Han, Lai, and
Li~\cite{HanLaiLi2018} proves the following sufficient condition for $\Z_3$-connectivity.
\begin{theorem}[Han--Lai--Li~\cite{HanLaiLi2018}]\label{thm:hll-4tree}
Every graph containing four edge-disjoint spanning trees is $\Z_3$-connected.
\end{theorem}

We also use the following rooted form of a theorem of Han, Lai, and
Li~\cite[Theorem~1.9(ii)]{HanLaiLi2018}.

\begin{theorem}[Han--Lai--Li~\cite{HanLaiLi2018}]\label{thm:sec4-hll-rooted}
Let $G$ contain four pairwise edge-disjoint spanning trees. Let
$z\in V(G)$ satisfy $d_G(z)\leq 7$, and assume that $G-z$ is
$2$-edge-connected. Then every valid preorientation at $z$ extends to a modulo $3$-orientation of $G$.
\end{theorem}

\subsection{Essential edge connectivity}

The following partition lemma is the core of the proof of
\cref{thm:essential-23}.

\begin{lemma}\label{lem:sec4-ore-packing}
Let $J$ be a graph with a distinguished vertex $z$.
Assume that
\begin{enumerate}[(i)]
 \item $J$ is $4$-edge-connected and essentially
       $21$-edge-connected;
 \item $4\leq d_J(z)\leq 7$, and $d_J(v)\geq 5$ for every
       $v\neq z$;
 \item $d_J(u)+d_J(v)\geq 25$ for every edge $uv\in E(J)$.
\end{enumerate}
Then $J$ contains four pairwise edge-disjoint spanning trees.
\end{lemma}

\begin{proof}
Let $\mathcal P=\{V_1,\ldots,V_r\}$ be a partition of $V(J)$, and
let $Q$ be obtained by contracting every part and deleting the
resulting loops. Let
\[
 |E(Q)|=e_J(\mathcal P)
 \quad\text{and}\quad
 d_Q(V_i)=d_J(V_i).
\]
The assertion is immediate when $r=1$. When $r=2$, it follows from
the $4$-edge-connectivity of $J$. Assume that $r\geq 3$.

We first consider one exceptional configuration. Suppose that
$J[V_i]$ contains an edge and $J-V_i$ is edgeless. If $z\in V_i$,
then
\[
 e_J(\mathcal P)
 =\sum_{j\neq i}d_Q(V_j)\geq 5(r-1).
\]
If $z\notin V_i$, and at most one part outside $V_i$ contains $z$, then
\[
 e_J(\mathcal P)\geq 4+5(r-2)=5r-6.
\]
Both bounds are at least $4(r-1)$. We may therefore assume that the
complement of every edge-containing part also contains an edge.

Notice that every vertex of $Q$ has degree at least $4$. Moreover, we claim that at most one
vertex of $Q$ has degree $4$. Indeed, assume that $d_Q(V_i)=4$.
Since this cut is not essential, one shore is independent. If $V_i$
is independent and does not contain $z$, then $d_Q(V_i)\geq 5$. If
$V_i$ contains $z$ and another vertex, then $d_Q(V_i)\geq 4+5$.
Thus the only possibility on this shore is $V_i=\{z\}$ and
$d_J(z)=4$. If the complementary shore is independent, then it
contains vertices in at least two parts, and its boundary is at least
$4+5$. This is again a contradiction.

We next show that every edge $V_iV_j$ of $Q$ satisfies
\begin{equation}\label{eq:sec4-quotient-ore}
 d_Q(V_i)+d_Q(V_j)\geq 25.
\end{equation}
If one of the two parts contains an edge, its boundary is essential
and has size at least $21$; the other boundary has size at least $4$.
If both parts are independent, let $xy$ be an edge of $J$ with
$x\in V_i$ and $y\in V_j$. Then
\[
 d_J(x)\leq d_Q(V_i)
 \quad\text{and}\quad
 d_J(y)\leq d_Q(V_j),
\]
and \eqref{eq:sec4-quotient-ore} follows from assumption (iii).

We shall complete the proof of this lemma by applying discharging method.
Give each vertex of $Q$ initial charge equal to its degree. A vertex
of degree at least $9$ sends
\[
 \frac{d_Q(v)-8}{d_Q(v)}
\]
along each incident edge and therefore retains charge $8$. By
\eqref{eq:sec4-quotient-ore}, a vertex of degree
$i\in\{5,6,7\}$ finishes with charge at least
\[
\begin{array}{c|c}
 i&\text{final charge}\\ \hline
 5&5+5(20-8)/20=8,\\
 6&6+6(19-8)/19=180/19>8,\\
 7&7+7(18-8)/18=98/9>8.
\end{array}
\]
A vertex of degree $8$ retains charge $8$. If the possible
degree-$4$ vertex occurs, all its neighbors have degree at least $21$,
so its final charge is at least
\[
 4+4(21-8)/21=\frac{136}{21}>0.
\]
Consequently,
\[
 2|E(Q)|>8(r-1).
\]
Since $|E(Q)|$ is an integer, $|E(Q)|\geq 4(r-1)$. The partition
$\mathcal P$ is arbitrary, so \cref{thm:sec4-nwt} completes the
proof.
\end{proof}

We now prove the rooted statement needed for
\cref{thm:essential-23}. 

\begin{theorem}\label{thm:sec4-essential-rooted}
Let $q\geq 23$. Let $G$ be a simple, $4$-edge-connected, essentially
$q$-edge-connected graph, and let $z\in V(G)$ satisfy
$d_G(z)\leq 7$. Every valid preorientation at $z$ extends to a
modulo $3$-orientation of $G$.
\end{theorem}

\begin{proof}
Let
\[
 S=\{s\in V(G)\setminus\{z\}:d_G(s)=4\}.
\]
By \cref{prop:degree-sum}, the set $S$ is independent, every neighbor
of a vertex in $S$ has degree at least $q-2\geq 21$, and no vertex of
$S$ is adjacent to $z$.

We completely lift every vertex of $S$. Consider an unprocessed
vertex $s\in S$. Previous liftings add edges only between vertices
outside $S$, so the four edges incident with $s$ and its four distinct
neighbors remain unchanged. The vertex $s$ is nonseparating. Indeed,
if $G-s$ had two components, each component would send at least four
edges to $s$, contrary to $d_G(s)=4$. By
\cref{thm:sec4-mader}, a complete lifting at $s$ preserves every
local edge-connectivity between the remaining vertices. Repeating
this operation gives a graph $J$ on
$V(G)\setminus S$ such that
\begin{align}
 \lambda_J(x,y)&=\lambda_G(x,y)
 &&(x\neq y\in V(J)), \label{eq:sec4-local-preserved}\\
 d_J(x)&=d_G(x)
 &&(x\in V(J)). \label{eq:sec4-degree-preserved}
\end{align}
In particular, $J$ is $4$-edge-connected,
$4\leq d_J(z)\leq 7$, and every other vertex has degree at least $5$.

For an edge created by a lifting,
both ends have degree at least $21$. 
For an original edge, it follows from \cref{prop:degree-sum} and \eqref{eq:sec4-degree-preserved} that the sum of degrees of two ends is at least $25$. 
Hence every edge $uv$ of $J$ satisfies
\begin{equation}\label{eq:sec4-split-ore}
 d_J(u)+d_J(v)\geq 25.
\end{equation}

\begin{claim}\label{cl:sec4-split-essential}
The graph $J$ is essentially $21$-edge-connected.
\end{claim}

\begin{claimproof}
We first show that no edge-cut of size at most $20$ has at least two
edges on each shore. Suppose that $E_J(A,A^c)$ is such an edge-cut.
If both shores contain a vertex of degree at least $21$, let
$x\in A$ and $y\in A^c$ be two such vertices. By
\eqref{eq:sec4-local-preserved},
$\lambda_G(x,y)\leq 20$. A minimum $x$--$y$ cut in $G$ may be
selected as an edge-cut. Since its size is less than $q$, it is not
essential and therefore has a singleton shore. This shore must be
$\{x\}$ or $\{y\}$, whose boundary has size at least $21$, a
contradiction.

Consequently, one shore, say $A$, contains only vertices of degree at
most $20$. No vertex of $A$ is adjacent in $G$ to a vertex of $S$.
Thus no lifting changes an edge incident with $A$, and
\begin{equation}\label{eq:sec4-cut-unchanged}
 E_J(A,A^c)=E_G(A,V(G)\setminus A)
 \quad\text{and}\quad
 J[A]=G[A].
\end{equation}
Both shores contain an edge in $G$. Indeed, an added edge on the
complementary shore of $J$ is the image of a two-edge path whose
internal vertex belongs to $S$, and the entire path lies on the same
shore in $G$. Hence \eqref{eq:sec4-cut-unchanged} gives an essential
edge-cut of $G$ of size at most $20$, a contradiction.

Now suppose that $J$ has an essential edge-cut of size at most $20$.
By the preceding paragraph, one shore contains exactly one edge
$uv$. Since the shore is connected, it consists of $u$ and $v$ and
$uv$ is its only edge. By \eqref{eq:sec4-split-ore}, the boundary of
$\{u,v\}$ has size at least $25-2=23$, a contradiction. This proves
the claim.
\end{claimproof}

\begin{claim}\label{cl:sec4-root-deletion}
The graph $J-z$ is $2$-edge-connected.
\end{claim}

\begin{claimproof}
If $J-z$ were disconnected, two of its components would each send at
least four edges to $z$. This would give $d_J(z)\geq 8$. Hence
$J-z$ is connected.

Suppose that $J-z$ has a bridge with shores $A$ and $B$. The graph
$J-z$ has at least four vertices. If $S=\varnothing$, this follows
from the simplicity and $4$-edge-connectivity of $G$. If
$S\neq\varnothing$, every vertex of $S$ has four distinct neighbors
in $V(J)\setminus\{z\}$. We may therefore assume that $|A|\geq 2$.
Then $J[A]$ contains an edge. The other shore together with $z$ also
contains an edge. This is clear when $|B|\geq 2$; when $B=\{b\}$,
the $4$-edge-connectivity of $J$ forces an edge between $b$ and $z$.
By \cref{cl:sec4-split-essential},
\[
 d_J(A)\geq 21
 \quad\text{and}\quad
 d_J(B)\geq 4.
\]
The bridge is counted in both boundaries. Thus
\[
 d_J(z)=d_J(A)+d_J(B)-2\geq 23,
\]
contrary to $d_J(z)\leq 7$. This proves the claim.
\end{claimproof}

By \cref{cl:sec4-split-essential}, \eqref{eq:sec4-split-ore}, and the
degree properties above, the graph $J$ satisfies all hypotheses of
\cref{lem:sec4-ore-packing}. It therefore contains four pairwise
edge-disjoint spanning trees. By \cref{cl:sec4-root-deletion} and
\cref{thm:sec4-hll-rooted}, the valid orientation at $z$ extends
to a modulo $3$-orientation of $J$.

Finally, reverse the complete liftings. Replace each oriented added
edge with the corresponding directed two-edge path through the
restored vertex. Each restored vertex receives one entering and one
leaving edge from each lifted pair. Hence every restored vertex is
balanced, all previous imbalances remain unchanged, and the
preorientation at $z$ is retained. This gives the required modulo
$3$-orientation of $G$.
\end{proof}

\begin{proof}[Proof of \Cref{thm:essential-23}]
For otherwise, suppose that $G$ is a counterexample with $|V(G)|$ minimum. If $G$ contains a nontrivial
$\Z_3$-connected subgraph $K$, then $G/K$ remains
$4$-edge-connected and essentially $23$-edge-connected by
\cref{prop:contraction-connectivity}. Minimality and
\cref{prop-extension} together give a modulo $3$-orientation of $G$.
Thus $G$ is $\Z_3$-reduced.

By part (i) of \cref{lem:basic-facts}, the graph $G$ has no parallel
edges. Hence $G$ is simple. If $\delta(G)\geq 6$, then $G$ is
$6$-edge-connected. By \cref{thm:ltwz-six}, the graph $G$ is $\Z_3$-connected, a
contradiction.

We may therefore let $z$ be a vertex with
$4\leq d_G(z)\leq 5$. The edges incident with $z$ clearly admit a valid
preorientation. By \cref{thm:sec4-essential-rooted}, this
preorientation extends to a modulo $3$-orientation of $G$, a
contradiction. The equivalence between modulo
$3$-orientations and nowhere-zero $3$-flows completes the proof.
\end{proof}

\subsection{Forbidden edge-cuts}

For an integer $q\geq 6$, let $\mathcal C_q$ denote the class of
$4$-edge-connected graphs in which every edge-cut has
size $4$, $5$, or at least $q$. Thus the hypothesis of
\cref{thm:forbidden-20} is exactly membership in $\mathcal C_{21}$.
The next lemma is a refined partition count. The important point is
that the edges at the root cause a total loss of at most $5$ over the
whole partition.

\begin{lemma}\label{lem:sec4-rooted-packing}
Let $J$ be a graph with partition $V(J)=H\mathbin{\cup}L\mathbin{\cup}\{z\}$.
Assume that $H\neq\varnothing$, $4\leq d_J(z)\leq 5$, and every
neighbor of $z$ belongs to $H$. Assume also that $L$ is independent,
every vertex of $L$ has degree $5$, and
\[
 \lambda_J(x,y)\geq 21
 \qquad\text{for all distinct }x,y\in H.
\]
Then $J-z$ contains four pairwise edge-disjoint spanning trees.
\end{lemma}

\begin{proof}
Let $Q=J-z$. Suppose that a partition $\mathcal P$ of $V(Q)$ violates
the Nash-Williams--Tutte inequality for four spanning trees.
Let $4(|\mathcal P|-1)-e_Q(\mathcal P)$ denote the deficiency of partition $\mathcal P$.
Among all
such partitions, select one with deficiency maximized. And subject to this choice, select one with as many parts as possible.

Let $v\in L$ belong to a nonsingleton part. If $v$ is separated as a
singleton part, the number of parts increases by one, while the
number of crossing edges increases by, say, $h$, the number of edges
from $v$ to the remaining vertices of its old part. The new
deficiency is larger by $4-h$. Hence $h\geq4$. If $h=4$, the
deficiency is unchanged and the new partition has more parts. It
follows that $h=5$. Thus all five edges of $v$ remain in its old
part. Moreover, every part containing only vertices of $L$ is a
singleton, since $L$ is independent.

Let $r$ denote the number of parts meeting $H$, and let $s$ denote
the number of the remaining parts. The latter are singleton vertices
of $L$. If $r=1$, then
\[
 e_Q(\mathcal P)=5s\geq 4s=4(|\mathcal P|-1),
\]
a contradiction. Assume that $r\geq 2$, and let $a$ be the number of
edges joining two distinct parts that meet $H$. Each part meeting
$H$ separates two vertices of $H$ and therefore has boundary at least
$21$ in $J$. Summing these $r$ boundary sizes gives
\[
 2a+5s+d_J(z)\geq 21r.
\]
Since $d_J(z)\leq 5$,
\begin{equation}\label{eq:sec4-global-root-loss}
 2a+5s\geq 21r-5.
\end{equation}
Also,
\[
 e_Q(\mathcal P)=a+5s.
\]
If $s\geq 4r-4$, then
\[
 e_Q(\mathcal P)\geq 5s
 \geq 4r+4s-4.
\]
If $s\leq 4r-5$, then \eqref{eq:sec4-global-root-loss} gives
\begin{align*}
 e_Q(\mathcal P)-4(r+s-1)
 &\geq \frac{21r-5+5s}{2}-4r-4s+4\\
 &=\frac{13r-3s+3}{2}>0.
\end{align*}
Both cases contradict the choice of $\mathcal P$. Therefore every
partition satisfies the required inequality, and
\cref{thm:sec4-nwt} completes the proof.
\end{proof}

\begin{theorem}\label{thm:sec4-gap-rooted}
Let $q\geq 21$, $G\in\mathcal C_q$, and let $z\in V(G)$ have
degree $4$ or $5$. Every valid preorientation at $z$ extends to a
modulo $3$-orientation of $G$.
\end{theorem}

\begin{proof}
We argue by induction on $|V(G)|$. If $|V(G)|=2$, the
preorientation balances the other vertex automatically. Suppose that
$|V(G)|=3$, and denote the other two vertices by $u$ and $v$. There
are at least two $uv$-edges. Otherwise, the $4$-edge-connectivity of
$G$ would force at least three edges from each of $u$ and $v$ to
$z$, contrary to $d_G(z)\leq 5$. By part (i) of
\cref{lem:basic-facts}, the multiple edges between $u$ and $v$ can realize the imbalance
required at $u$ and $v$. Thus the assertion holds.

Suppose that $G$ has a nontrivial edge-cut
$B=E_G(X,X^c)$ of size $4$ or $5$, where $z\in X$. Let
\[
 G_1=G/X^c
 \quad\text{and}\quad
 G_2=G/X.
\]
By \cref{prop:contraction-connectivity}, both graphs belong to
$\mathcal C_q$ and have fewer vertices. Apply induction to $G_1$ at
$z$. The resulting orientation induces a valid preorientation of
$B$ at the contracted vertex of $G_2$. Apply induction to $G_2$ with
this preorientation. The two orientations agree on $B$ and combine
to give the required orientation of $G$.

We may therefore assume that every edge-cut of size $4$ or $5$ is
trivial. 
We can also assume $G-z$ is $\Z_3$-reduced, i.e., it contains no nontrivial $\Z_3$-connected subgraph. For otherwise, suppose $H$ is a nontrivial $\Z_3$-connected subgraph of $G-z$ and we consider the graph $G/H$. By \cref{prop:contraction-connectivity} and induction, the valid preorientation at $z$ extends to a modulo $3$-orientation of $G/H$.
It follows from \cref{prop-extension} and part (i) of \cref{lem:basic-facts} that $G$ also admits a modulo $3$-orientation.
In particular, we can assume no two vertices outside $z$ are joined by parallel
edges (i.e., $G-z$ is a simple graph).

\begin{claim}\label{cl:sec4-atomic-structure}
Every vertex $v\neq z$ has degree $4$, degree $5$, or degree at least
$q$. Moreover, all vertices of degree $4$ or $5$, including $z$, form
an independent set.
\end{claim}

\begin{claimproof}
The first part of this lemma is trivial since $|E(w,G-w)|=d(w)$, and for each vertex $w$, the edge-cut $E(w,G-w)$ has size $4$, $5$, or at least $q$.

Suppose next that two vertices $u$ and $v$ of degree $4$ or $5$ are
adjacent. Then
\begin{equation}\label{eq:sec4-low-pair-cut}
 d_G(\{u,v\})
 =d_G(u)+d_G(v)-2e_G(u,v)\leq 8.
\end{equation}
If $G-\{u,v\}$ is connected, this is a nontrivial edge-cut of a
forbidden size or a nontrivial edge-cut of size $4$ or $5$, a
contradiction. Otherwise, let $C_1,\ldots,C_t$ be the components of
$G-\{u,v\}$. The orders of the edge-cuts around these components
sum to the left-hand side of
\eqref{eq:sec4-low-pair-cut}. Hence $t=2$, both edge-cuts have size
$4$, $d_G(u)=d_G(v)=5$, and $e_G(u,v)=1$. Since the two edge-cuts
are trivial, both components are singletons. Then we divide the proof according to whether $z\in\{u,v\}$.

If neither $u$ nor $v$ is $z$, then at least one of these singletons is
different from $z$ and is joined to $u$ and $v$ by four edges. This is a contradiction since $G-z$ is $\Z_3$-reduced. If, say,
$u=z$, then each singleton sends at most one edge to $v$ and at least
three edges to $z$. Together with the edge $zv$, this gives
$d_G(z)\geq 7$, contrary to $d_G(z)\leq 5$. This proves the claim.
\end{claimproof}

Let $S$ be the set of degree-$4$ vertices different from $z$, let
$L$ be the set of degree-$5$ vertices different from $z$, and let
\[
 H=V(G)\setminus(S\cup L\cup\{z\}).
\]
By \cref{cl:sec4-atomic-structure}, every vertex of $H$ has degree at
least $q$, and the sets $S\cup L\cup \{z\}$ is independent. In particular, every neighbor of $z$ belongs to $H$.

We completely lift the vertices of $S$ in succession. By
\cref{thm:sec4-mader}, each step preserves all local
edge-connectivities among the remaining vertices. The operation is
always available. Indeed, the current graph remains
$4$-edge-connected. If an unprocessed vertex $s\in S$ were
separating, two components of the graph minus $s$ would each send at
least four edges to $s$, contrary to $d_G(s)=4$. Denote the resulting
graph by $J$. The liftings do not change an edge incident with $z$
or with a vertex of $L$. Thus
\[
 V(J)=H\mathbin{\cup}L\mathbin{\cup}\{z\},
\]
the set $L$ is independent, and every vertex of $L$ still has degree
$5$. For distinct vertices $x,y\in H$, a minimum $x$--$y$ cut in
$G$ may be selected as an edge-cut. An edge-cut of size less than
$q$ has size $4$ or $5$ and is trivial, so it cannot separate two
vertices of $H$. Therefore,
\[
 \lambda_J(x,y)=\lambda_G(x,y)\geq q\geq 21.
\]

By \cref{lem:sec4-rooted-packing}, the graph $J-z$ contains four
pairwise edge-disjoint spanning trees. Hence $J-z$ is
$\Z_3$-connected by \cref{thm:hll-4tree}. The fixed directions at
$z$ induce a $\Z_3$-boundary on $J-z$. Orient $J-z$ to cancel this
boundary. Together with the valid preorientation at $z$, this gives a
modulo $3$-orientation of $J$.

Reverse all complete liftings. Each oriented added edge becomes a
directed two-edge path through the restored vertex, so every restored
vertex remains balanced. The resulting orientation of $G$ extends
the valid preorientation at $z$ and completes the induction.
\end{proof}

\begin{proof}[Proof of \Cref{thm:forbidden-20}]
The hypothesis is equivalent to $G\in\mathcal C_{21}$. We argue by
induction on $|V(G)|$. The assertion is immediate for graphs on at
most two vertices.

Suppose that $G$ has a nontrivial edge-cut $B$ of size $4$ or $5$.
Both quotient graphs have fewer vertices and belong to
$\mathcal C_{21}$ by \cref{prop:contraction-connectivity}. Apply induction
to one quotient. Its orientation induces a valid preorientation of
$B$ at the contracted vertex of the other quotient. The rooted
extension theorem,
\cref{thm:sec4-gap-rooted}, extends the induced preorientation. The
two orientations agree on $B$ and combine to give a modulo
$3$-orientation of $G$.

We may therefore assume that every edge-cut of size $4$ or $5$ is
trivial. If $G$ has a vertex $z$ of degree $4$ or $5$, select a valid
preorientation at $z$ and apply \cref{thm:sec4-gap-rooted}. If no
such vertex exists, then $G$ has no edge-cut of size $4$ or $5$.
Every edge-cut has size at least $21$, so $G$ is $6$-edge-connected.
By \cref{thm:ltwz-six}, it admits a modulo $3$-orientation. This completes the proof.
\end{proof}

\end{document}